\documentclass[12pt,a4paper]{amsart}
\usepackage[utf8]{inputenc}
\usepackage[body={6.5in, 8.6in},left=1.3in,right=1.3in]{geometry}

\usepackage[utf8]{inputenc}
\usepackage[english]{babel}
\usepackage{amsmath}
\usepackage{amsthm}
\usepackage[shortlabels]{enumitem}
\usepackage{amssymb}
\usepackage{amsfonts}
\usepackage{mathrsfs}
\usepackage{setspace}
\usepackage{accents}
\usepackage[dvipsnames]{xcolor}
\usepackage{mathtools}
\usepackage{bbm}
\usepackage{tikz}
\usepackage{tikz-cd}
\usepackage{quiver}
\usepackage{xfrac}
\usepackage{faktor}
\usepackage{marvosym}
\usepackage{bm} 
\usepackage[normalem]{ulem}
\usepackage{hyperref}
\hypersetup{
    colorlinks=true,
    linkcolor=blue,
    filecolor=magenta,      
    urlcolor=cyan,
    pdftitle={Overleaf Example},
    pdfpagemode=FullScreen,
    }

\newtheorem{lem}{Lemma}[section]
\newtheorem{thm}[lem]{Theorem}
\newtheorem*{thm*}{Theorem}
\newtheorem{cor}[lem]{Corollary}
\newtheorem{prop}[lem]{Proposition}

\theoremstyle{definition}
\newtheorem{df}[lem]{Definition}
\newtheorem{remark}[lem]{Remark}

\newtheorem{question}[lem]{Question}

\theoremstyle{remark}
\newtheorem*{remarks*}{Remarks}
\newtheorem*{note*}{Note}

\newcommand{\accapo}{\(\rule[-25pt]{0pt}{25pt}\)\noindent}
\newcommand{\accapino}{\(\rule[-10pt]{0pt}{10pt}\)\noindent}

\newcommand{\N}{\mathbb{N}}

\newcommand{\Z}{\mathbb{Z}}
\newcommand{\Q}{\mathbb{Q}}
\newcommand{\R}{\mathbb{R}}
\newcommand{\C}{\mathbb{C}}

\newcommand{\proj}{\mathbb{P}}

\newcommand{\Dcal}{\mathcal{D}}

\newcommand{\Lcal}{\mathcal{L}}

\newcommand{\Ocal}{\mathcal{O}}
\newcommand{\cO}{\mathcal{O}}

\newcommand{\Xcal}{\mathcal{X}}

\DeclareMathOperator{\Pic}{\textup{Pic}}
\DeclareMathOperator{\NS}{\textup{NS}}

\DeclareMathOperator{\codim}{\textup{codim}}

\DeclareMathOperator{\Ker}{\textup{Ker}}

\DeclareMathOperator{\Hom}{\textrm{Hom}}
\DeclareMathOperator{\End}{\textrm{End}}

\newcommand{\simsigma}{\widetilde{\sigma}}
\newcommand{\simvarphi}{\widetilde{\varphi}}

\newcommand{\simS}{\widetilde{S}}
\newcommand{\simeta}{\widetilde{\eta}}
\newcommand{\virgolette}{``}

\title{Lefschetz Defect in Families}
\author{Matteo Verni}
\address{Sorbonne Université, Université Paris Cité, CNRS, IMJ-PRG, F-75005 Paris, France}
 \email{{\tt matteo.verni@imj-prg.fr}}

\begin{document}
\begin{abstract}
     We provide a cohomological characterization of the Lefschetz defect of smooth complex projective varieties. As a consequence, we deduce that the Lefschetz defect of a smooth Fano variety is invariant under smooth deformation. We also characterize the Lefschetz defect of an abelian variety in terms of its isogeny factors.
 \end{abstract}
 
\maketitle
 
\section{Introduction}

In \cite{cas2012}, the author introduces the \textit{Lefschetz defect} \(\delta_X\) (see Definition \ref{df_defect}) for \(X\) a smooth Fano variety. It  measures how many \virgolette special" steps occur in runs of the MMP of \((X,-D)\) for all effective divisors \(D\). This invariant has proven useful in the effort to classify Fano fourfolds: for the latest developments in this direction, see \cite{cas2024}. 

Since there are finitely many deformation families of Fano manifolds in each dimension, to classify these varieties one is particularly interested in those invariants that are preserved under deformation. One is then naturally led to ask whether the Lefschetz defect is such an invariant. Previously, this had been established in the case \(\delta_X=2\) and \(\rho_X=3\) as an indirect consequence of the classification result of \cite{Secci2023}.

In the present work, our main result is that the Lefschetz defect is indeed deformation invariant.

\begin{thm}\label{thm main}
Let \(X_1\) and \(X_2\) be two smooth Fano varieties which are deformation equivalent via a smooth deformation. Then \(\delta_{X_1}=\delta_{X_2}.\)
\end{thm}
It follows that the strong characterization results of \cite{cas2012}, \cite{cas_dru}, \cite{cas_romano_secci2022} and \cite{cas2024} hold throughout each deformation family. 

On the way to the proof we make the following remark of independent interest (see Section \ref{subsec_basic} for definitions).
\begin{prop}\label{prop_invariance_rat_eq}
For any linearly equivalent prime divisors \(D\) and \(D'\) on a smooth projective variety, we have
\[\delta_X(D)=\delta_X(D').\]
\end{prop}

In the second part of the paper, we begin to study the Lefschetz defect of smooth projective varieties other than Fano. The second main result precisely relates the Lefschetz defect of an abelian variety with its isogeny factors and their multiplicity. We denote two abelian varieties \(A,A'\) being isogenous by \(A\sim A'\).

\begin{thm}\label{thm_defect_of_abelians}
Let \(A\) be an abelian variety of dimension \(n\). Then \(\delta_A=0\) if and only if \(A\) does not contain elliptic curves nor abelian surfaces of Picard number greater than 1. If instead \(\delta_A>0\), then exactly one of the following is verified:
\begin{itemize}
\item \(A\sim C \times E\) where \(E\) is an elliptic curve, and if \(k_E \geq 1\) is its multiplicity as an isogeny factor of \(A\) then \[\delta_A=\begin{cases}
k_E & \mathrm{ \ if \ } E \textrm{ does not have CM }\\
2k_E -1 & \mathrm{ \ if \ } E \textrm{ has CM}
\end{cases}.\]
\item \(A\sim C\times S\) where \(S\) is a simple abelian surface of type II, \(\rho_S=3\) and \[\delta_A=2.\]
Furthermore, any elliptic curve showing up as an isogeny factor of \(A\) does so with multiplicity at most one.
\item \(A\sim C\times S'\) where \(S'\) is a simple abelian surface of type I or IV, \(\rho_{S'}=2\) and \[\delta_A=1.\]  Furthermore, there are no elliptic curves contained in \(A\).
\end{itemize} 

\end{thm}

This result can also be seen as one possible way of addressing the problem brought up in the introduction of \cite[Section 8]{cas2022}, namely to find splitting results in the spirit of \cite[Th. 3.3.]{cas2012},  outside the smooth Fano setting. 
This was done for mildly singular Fano varieties in \cite{Della_Noce_2012}, and our result provides an instance completely outside the Fano context.

\subsection{Outline of the article}

In Section \ref{sec_defect} we recall basic definitions and remarks about the Lefschetz defect. In Section \ref{subsec_another_char} we prove a generalization  (Proposition \ref{prop_generalized_voisin_lemma}) of \cite[Lemme 1.5]{voisin1992}, from which we deduce that \(\delta_X(D)\) depends only on the cohomological class \([D]\) and not on \(D\) itself. This provides a purely cohomological characterization of \(\delta_X\) to be used throughout, and it proves Proposition \ref{prop_invariance_rat_eq}.

In Section \ref{sec_Fano_in_families} we collect the supporting results and give the proof of Theorem \ref{thm main}: the necessary ingredients are Corollary \ref{cor_another_N1defect} and a result on deformation invariance of the cone of effective divisors, which was established at increasing levels of generality in the works \cite{Siu2002}, \cite{deFernex_Hacon2011} and \cite{Totaro2012}.

In Section \ref{sec_defect_abelian} we make the link between Lefschetz defect and isogeny factors of an abelian variety, and prove Theorem \ref{thm_defect_of_abelians}. 

\subsection*{Acknowledgements}
I would first like to thank my advisors Claire Voisin and Emanuele Macrì, for sharing their knowledge and for their invaluable advice on how to write an article. In particular, I thank Claire Voisin for the suggestion of adapting her result of \cite{voisin1992} to the study of \(\delta_X\). I would like to thank Saverio Secci for explaining to me the previously known case of the invariance result, and also Angel Rios-Ortiz and Enrico Lampetti for helpful conversations.

This work was supported by the ERC Synergy Grant HyperK (Grant agreement No. 854361)

\section{Lefschetz Defect}\label{sec_defect}
\subsection{Basic definitions and remarks}\label{subsec_basic}

In this section, \(X\) will always denote a smooth projective variety, if not specified otherwise. We first recall general definitions and remarks, and then establish the characterization (\ref{eq another N1 defect}).
\accapino

A \textit{prime} divisor is by definition the Weil divisor defined by one integral closed subvariety of codimension one. 

For any algebraic variety \(Y\), \(N_1(Y)_\Q\) denotes the \(\Q\)-vector space generated by all integral curves \(C \subset Y\) modulo numerical equivalence. Proper morphisms \(Y \rightarrow Z\) functorially induce linear maps \(N_1(Y)_\Q \rightarrow N_1(Z)_\Q\) in the obvious way. 

\begin{df}\label{df_defect}
Let \(i: D \hookrightarrow X\) be the embedding of a prime divisor.
\noindent We define \[\delta_X(D):= \codim_{N_1(X)_{\Q}} i_* N_1(D)_{\Q}.\] 

We define the \textbf{Lefschetz defect of \(X\)} as the number
\begin{equation}\label{eq_definition_delta_X}
    \delta_X := \max_{\substack{D\subset X \\\textrm{ prime}}} \delta_X(D).
\end{equation}
\end{df}
\noindent For a detailed survey on the Lefschetz defect and its applications, see \cite{cas2022}.

We observe that we could have taken the maximum in (\ref{eq_definition_delta_X}) over all (possibly nonreduced, possibly reducible) effective divisors. Firstly, for any nonreduced divisor
\(D\), one has \(\delta_X(D)=\delta_X(D_{red})\).
Indeed, when computing \(\dim_\Q i_* N_1(D)_\Q\) we are considering integral curves inside \(X\): since the embedding of a reduced subscheme only depends on the topological embedding, the class of a curve \(C \subset D\) does not care about the nonreduced structure of \(D\), the only thing that matters is the topological subspace \(|C| \subset |D| \subset |X|\). 
Moreover, if \(D=D_1 \cup D_2\), one sees that \(\delta_X(D)\leq \delta_X(D_1)\) since all curves lying on \(D_1\) are also lying on \(D\). 

We conclude that
\begin{equation}\label{remark alldivisors}
\delta_X=\max_{\substack{D\subset X \\\textrm{ effective}}} \delta_X(D).
\end{equation}

Denoting the Picard rank of \(X\) by \(\rho_X\), we remark that
\begin{equation}\label{range_for_defect}
0\leq \delta_X \leq \rho_X -1 .\end{equation}
Indeed, \(\delta_X \neq \rho_X\) since any integral curve \(C\) on a divisor \(D\) is intersected nontrivially by hyperplane sections of \(X\), so \([C]\) represents a nontrivial class of \(i_* N_1(D)_\Q\).

Another situation where \(\delta_X(D)\) is easily computed is the following.
Let \(D\subset X\) be as above and suppose there exists a retraction \(X \rightarrow D\), i.e., a morphism of varieties which is a left-inverse to the inclusion morphism \(D \hookrightarrow X\). Then \(X\rightarrow D\) induces on numerical \(1\)-cycles a retraction to the map \(N_1(D,\Q) \rightarrow N_1(X,\Q)\), which then must be injective. We conclude that
\begin{equation}\label{rk retraction}
\delta_X(D)=\rho_X - \rho_D.
\end{equation}
\begin{remark}\cite[Remark 3]{cas2022}
For any complex projective variety \(X\), \(N_1(X)^\vee_\Q\) 
is naturally identified with \(N^1(X)_{\Q}\). Moreover the latter injects into \(H^2(X,\Q)\),
as homological equivalence of divisors is equivalent to numerical equivalence up to torsion.
Thus we obtain a more cohomological formulation of \(\delta_X(D)\), namely
\begin{equation}\label{eq N1 formula}
    \delta_X(D)= \dim  \Ker \big(\eta^* : N^1(X)_\Q \rightarrow H^2(D,\Q)\big).
\end{equation}

Note also that in case \(H^2(X,\Ocal_X)=0\), we have the diagram

\[\begin{tikzcd}
	{N^1(X)_\Q} & {N^1(D)_\Q} \\
	{H^2(X,\Q)} & {H^2(D,\Q)}
	\arrow[from=1-1, to=1-2]
	\arrow["\sim"', from=1-1, to=2-1]
	\arrow[hook, from=1-2, to=2-2]
	\arrow[from=2-1, to=2-2]
\end{tikzcd},\]
which implies 

\begin{equation*}
\delta_X(D)= \textrm{dim}  \Ker \big(\eta^* : H^2(X,\Q)\rightarrow H^2(D,\Q)\big).\end{equation*}
\end{remark}

\subsection{Another cohomological characterization of \(\delta_X(D)\)}\label{subsec_another_char}

For \(X\) a compact complex manifold, recall that there exists a \textbf{cycle class map}
\[CH^k(X)\rightarrow H^{2k}(X,\Q)\]
\[Z\mapsto [Z]\]

Recall that for any closed analytic subvariety \(\sigma : V\hookrightarrow X\), \([V]\) can be defined via any resolution of singularities \(\widetilde{V} \xrightarrow{\tau} V\) by pushing forward the fundamental class. More precisely, setting \(\overset{\sim}{\sigma}:=\sigma\circ \tau\)
\[[V]=\sigma_* (\underline{1}_{\overset{\sim}{Z}})=PD_X^{-1}\circ (\overset{\sim}{\sigma}^*)^\vee \circ PD_{\overset{\sim}{Z}}(\underline{1}_{\overset{\sim}{Z}})\]
where \(PD_X : H^k(X,\Q)\rightarrow H^{2n-k}(X,\Q)^\vee\) are Poincaré duality morphisms and \[\underline{1}_{\overset{\sim}{Z}} \in H^0(\overset{\sim}{Z},\Q)\simeq \Q\] is the generator associated to a given choice of orientation.
\accapino

Remarkably, the subspace of cohomology classes on a K\"ahler manifold \(X\) that are trivial when restricted to a submanifold \(V\) depends only on the cohomology class \([V]\). More precisely, one has
\begin{prop}\cite[Lemme 1.5]{voisin1992}\label{prop_voisin1992}
Let \(Y\) be a K\"ahler manifold and \(W\) a submanifold of codimension \(c\). Then 

\begin{small}\begin{equation*}\label{another_characterization}
\Ker\Big(H^2(Y,\R)\xrightarrow{\sigma^*}H^2(W,\R)\Big)=
\Ker\Big(H^2(Y,\R) \xrightarrow{\smallsmile [W]} H^{2+2c}(Y,\R)\Big).
\end{equation*}
\end{small}
\end{prop}

Note that this statement fails for cohomology of higher degrees, as explained by Voisin in Appendix B to \cite[Remark B.3]{shenyin2022}.
\accapo

We can in fact prove a generalization of the above proposition which works for certain singular subvarieties. Its proof follows in part that of \cite{voisin1992}, with adequate adjustments.

\begin{prop}\label{prop_generalized_voisin_lemma}
    Let \(X\) be a smooth projective variety over \(\C\) of dimension \(n\) and \(\sigma: V\hookrightarrow X\) a locally complete intersection subvariety of dimension \(d\geq 2\) and codimension \(c=n-d\). 
    Then 
    \begin{small}
    \begin{equation}\label{another_formula}
    \Ker\Big(H^2(X,\Q)\xrightarrow{\sigma^*}H^2(V,\Q)\Big)=\Ker\Big(H^2(X,\Q) \xrightarrow{\smallsmile [V]} H^{2+2c}(X,\Q)\Big).
        \end{equation}
        \end{small}
\end{prop}

It follows immediately that we also have the equality
\begin{small}
\begin{equation}\label{eq another N1 formula}
\Ker\Big(N^1(X,\Q)\xrightarrow{\sigma^*}H^2(V,\Q)\Big)=
\Ker\Big(N^1(X,\Q) \xrightarrow{\smallsmile [V]} H^{2+2c}(X,\Q)\Big).
\end{equation}
\end{small}

Using (\ref{eq N1 formula}) we deduce from (\ref{eq another N1 formula}) the following. 
\begin{cor}\label{cor_another_N1defect}
Let \(X\) be a smooth complex projective variety and let \(D\subset X\) be a prime divisor. Then
\begin{equation}\label{eq another N1 defect}
    \delta_X(D)= \dim \Ker\Big(N^1(X,\Q) \xrightarrow{\smallsmile [D]} H^{4}(X,\Q)\Big).
\end{equation}
\end{cor}

In particular, Proposition \ref{prop_invariance_rat_eq} holds.

\begin{remark}\label{rk_thanks_referee}
Note that the above Corollary holds just as well for irreducible but non-reduced divisors, as \(\delta_X(D)=\delta(mD)\) and \(-\smallsmile [mD]=m (- \smallsmile [D])\) for all \(m \in \Z_{>0}\). On the other hand, as was correctly pointed out by an anonymous referee, in Corollary \ref{cor_another_N1defect} one cannot relax the assumption of \(D\) being irreducible; to produce a counterexample one could take \(X\) to be a del Pezzo surface with a conic bundle \(\pi \colon X\to \proj^1\) such that for some \(t\in \proj^1\) the fiber \(\pi^{-1}(t)\) is the union of two (\(-1\))-curves. Then \(\delta_X(\pi^{-1}(t))=\rho_X -2\), while 
\[\dim \Ker \Big( N^1(X,\Q) \xrightarrow{\smallsmile [\pi^{-1}(t)]} H^{4}(X,\Q)\Big)=\rho_X -1.\]

\end{remark}

\noindent For the proof of Proposition \ref{prop_generalized_voisin_lemma} we will use the following results on the Betti cohomology of complex algebraic varieties:

\begin{thm}\label{thm_LHT_singular}
For \(Y\) a locally complete intersection subvariety of \(\proj^N\), and \(H\cap Y\) a general hyperplane section of \(Y\), one has that
\[H^k(Y,\Q) \rightarrow H^k(H\cap Y,\Q)\]
is an isomorphism for \(k < dim(Y)-1\) and injective for \(k= dim(Y)-1\).
\end{thm}
\begin{proof}
This is the theorem \virgolette LHT for singular spaces" at page 24  of \cite{goresky1988}, applied to the locally complete intersection case.
\end{proof}

\begin{thm}\cite[Proposition 8.2.7]{Deligne1974}\label{thm_Deligne}
Let \(\varphi : Y \rightarrow X \) be a proper morphism with \(X\) smooth projective. Let \(\tau: \widetilde{Y} \rightarrow Y\) be a desingularization and let \(\simvarphi := \varphi \circ \tau :\widetilde{Y} \rightarrow X\). Then for any \(k\), 
\begin{small}\[\Ker \, \Big(\varphi^* : H^k(X,\Q) \rightarrow H^k(Y,\Q) \Big) = \Ker \, \Big( \widetilde{\varphi}^* : H^k(X,\Q) \rightarrow H^k(\widetilde{Y},\Q) \Big)  .\]
\end{small}
\end{thm}

\begin{proof}[Proof of Proposition \ref{prop_generalized_voisin_lemma}]
Consider a resolution of singularities \(\widetilde{V} \xrightarrow{\tau} V \) 
and denote the induced composite morphism \( \widetilde{V} \rightarrow X\) by \(\simsigma := \sigma \circ \tau\).
One inclusion in (\ref{another_formula}) is easy: by the projection formula, \[\smallsmile [V]=\simsigma_* \circ \simsigma^*=\simsigma_* \tau^* \sigma^* ,\]
thus 
\[\Ker\,  \sigma^* \subset \Ker \, ( \smallsmile [V]).\]

We are left to prove the opposite inclusion, which we break down into steps.

\textit{Step 1:} We first reduce to the case where \(V\) is a surface. Let \(\alpha \in \Ker( \smallsmile [V])\), and let \(h\) be the ample class associated to an ample divisor \(H\). We denote by \(S\) the surface we obtain from \(V\) by taking a general \(H\)-hyperplane section a total of \(d-2\) times. Let us denote by \(j\) the embedding of \(S\) in \(V\), and set \(\eta=\sigma\circ j\). We have \([S]=h^{d-2} \smallsmile [V]\), hence \[\alpha \smallsmile [S]=0.\]

We have the following commutative diagram
\[\begin{tikzcd}
	{H^2(X)} & {H^2(V)} \\
	{H^2(S)}
	\arrow["{\sigma^*}", from=1-1, to=1-2]
	\arrow["{\eta^*}"', from=1-1, to=2-1]
	\arrow["{j^*}", hook, from=1-2, to=2-1]
\end{tikzcd}\]
where \(j^*\) is injective by Theorem \ref{thm_LHT_singular}. In order to conclude that \(\sigma^*\alpha=0\) it thus suffices to prove the implication
\[\alpha \smallsmile [S]=0 \Rightarrow \eta^*\alpha =0\]
which amounts exactly to the \(d=2\) case.

\textit{Step 2:} Let \(\pi: \simS \rightarrow S\) be a resolution of singularities and let \(\simeta:=\eta \circ \pi\). In this step we prove that 
\[\alpha \smallsmile [S] =0 \Rightarrow  \simeta^*\alpha=0.\]
The argument will follow very closely the one of \cite{voisin1992}, which proves Proposition \ref{prop_voisin1992} for any \(d\). However, note that it would not work here if we had not reduced to \(d=2\).

Consider the usual nondegenerate quadratic form on \(H^2(\simS,\Q)\)
\[q(\gamma,\nu):=\int_{\simS} \gamma \smallsmile \nu .\]
For any \(\alpha,\beta \in H^2(X,\Q)\) we have\[q(\simeta^*\alpha,\simeta^*\beta):=\int_{X} \alpha \smallsmile \beta\smallsmile [S],\]
from which we see that \(\alpha \smallsmile [S]=0\) implies \( \simeta^* \alpha \in \Ker q_{|\mathrm{Im} \, \simeta^*}\).
This means it suffices to prove that the restriction \(q_{|Im \, \simeta^*}\) is nondegenerate.

The signature of \(q\) is described by the Hodge index theorem. The real cohomology \(H^2(\widetilde{S},\R)\) decomposes as 
\begin{equation*}\label{Hodge_decomp} H^2(\widetilde{S},\R)=(H^{2,0}(\widetilde{S})\oplus H^{0,2}(\widetilde{S}))_\R \oplus H^{1,1}(\widetilde{S})_\R ,\end{equation*}
where the direct sum is orthogonal for \(q\) and the subscript \(\R\) designates the subspace of real cohomology classes. \(q\) is positive definite on the first subspace and it has signature \((+,-,\dots , -)\) on the second subspace. The real cohomology \(H^2(X,\R)\) has a similar decomposition and the morphism \(\simeta^*\), being a morphism of Hodge structures, preserves such a decomposition. It follows that 
\[\mathrm{Im} \, \simeta^*= \mathrm{Im} \,
\simeta^{*(2,0)+(0,2)} \oplus \mathrm{Im} \, \simeta^{*(1,1)} .\]

Hence it suffices to prove that \(q_{|\mathrm{Im} \, \simeta^{*(2,0)+(0,2)}}\)
and \(q_{|\mathrm{Im} \, \simeta^{*(1,1)}}\) are nondegenerate. For the first one this is obvious because \(q_{(H^{2,0}(\widetilde{S})\oplus H^{0,2}(\widetilde{S}))_\R}\) is positive definite. For the second one, this follows from the fact that the image contains the big and nef class \(\simeta^* h \), which thus has positive self intersection, and from the signature of \(q\) on \(H^{1,1}(\widetilde{S})_\R\).

\textit{Step 3:} We are left to prove the implication 
\[\simeta^* \alpha=0 \Rightarrow \eta^*\alpha =0.\] But this is exactly the content of Theorem \ref{thm_Deligne}, hence we are done.
\end{proof}

\section{Lefschetz defect in families of Fano Manifolds}\label{sec_Fano_in_families}

In this section we focus on the Fano manifold case: it is restricted to this context that the Lefschetz defect has been introduced, because of its interpretation in terms of MMP runs (\cite[Section 3]{cas2022}). The advantage of the Fano hypothesis for our study of the behaviour in families is that the cohomology of these variety is more closely related to its Picard groups; the first varies well with deformations, while the second is the one to which the Lefschetz defect heavily depends on. 

Throughout this section, \(X\) will denote a smooth Fano variety. 
First, note that \(H^2(X,\Ocal_X)=0\), thus \(N^1(X)_\Q=H^2(X,\Q)\). This means (\ref{eq another N1 defect}) becomes
\begin{equation}\label{eq_another_H2_delta}
\delta_X(D)= \dim \Ker\Big(H^2(X,\Q) \xrightarrow{\smallsmile [D]} H^{4}(X,\Q)\Big)
\end{equation}
We consider a \textbf{smooth deformation family of \(X\) over \(B\)} : this means a smooth morphism \(f\colon \Xcal \rightarrow B\), with \(B\) connected, for which the relative canonical is an anti-ample line bundle relatively over the base, making each fiber a smooth Fano variety, together with a chosen point \(0\in B\) such that \(\Xcal_0\simeq X\). We denote such data as an arrow \(f \colon \Xcal \rightarrow (B,0)\).

\begin{question}
How does \(\delta_{\Xcal_b}\) vary with \(b\in B\) ?
\end{question}

For \(\rho_X=1\) or \(\dim X \leq 2\), \(\delta_X\) is trivially deformation invariant. The first is clear from the bounds \(0\leq \delta_X \leq \rho_X -1\) of  \eqref{range_for_defect}, and the fact that \(\rho_X\) stays constant for smooth deformations of varieties with \(h^{2,0}=0\). This also covers the case \(\dim X=1\) since then \(X \simeq \proj^1\). If instead \(X\) is a del Pezzo surface, then any integral curve \(C\) has \(N_1(C)_\R = \R [C]\) hence \(\delta_X= \delta_X(D)=\rho_X-1\)

\subsection{Extending line bundles}
Ideally, we want to spread effective divisors from one fiber to a whole family of effective divisors over the base. This in general is only possible up to basechanging the family: \textit{as long as the basechange is surjective}, this does not pose any problem towards our final goal of comparing specific fibers \(X_t\) and \(X_s\), as we do not care about which parameter space \(t\) and \(s\) live in.

The first step consists in spreading a line bundle on a fiber to a line bundle over the entire total space of the family. This can be done whenever the relative Picard scheme is smooth over its base. 

\begin{thm}\label{spreadPic}
Let \(f\colon \Xcal \rightarrow B\) be a smooth projective morphism with integral fibers with \(B\) smooth connected quasi-projective over a field \(k\), and assume \(h^2(\Xcal_b, \Ocal_{\Xcal_b})=0\) for all \(b\in B\).  Let \(0 \in B\) be a point and \(L \in \Pic_{\Xcal_0}\).
Then there exists a generically finite, surjective morphism \((B',0') \xrightarrow{\tau} (B,0)\) with \(B'\) smooth connected and a line bundle \(\Lcal\) on \(\Xcal_{B'} = \Xcal \times_B B'\) such that  \(\Lcal_{|\Xcal_{0'}}\simeq L\).
\end{thm} 
\begin{proof}
This is well known, following from the fact that by \cite[Theorem 9.4.8]{fga} the relative Picard scheme \(\Pic_{\Xcal/B}\xrightarrow{\pi} B\) exists and by  \cite[Proposition 9.5.19]{fga} it is smooth. 
Finding the line bundle \(\Lcal\) as requested amounts to finding a basechange \(B' \rightarrow B\) such that \(\pi_B' \colon  \Pic_{\Xcal_{B'}/B'} \rightarrow B'\) has a section passing through the fiber of the projection \(\Pic_{\Xcal_{B'}/B'} \rightarrow \Pic_{\Xcal/B}\) above the point \(L \in \Pic_{\Xcal/B}\). 
This can be done, for example by looking at the Henselianization of the local ring at \(0 \in B\), from which one can lift morphisms along any smooth morphism. This lifting property then holds for a smaller, quasi-finite étale neighbourhood of \(U \rightarrow B\) of \(0\in B\), and one concludes by passing to a Nagata compactification \(U \subset B' \rightarrow B\), since a flat finitely presented proper map between integral schemes is surjective. 
\end{proof}

\subsection{Effectivity of Cartier divisors in a family}
By the previous section we can extend (up to basechange) a line bundle \(L\) on one fiber \(\Xcal_0\) to a line bundle \(\Lcal\) over the entire domain \(\Xcal\). What is not at all clear a priori is whether or not the restrictions \(\Lcal_{|\Xcal_b}\) will be effective in case \(L\) was effective. Luckily, even more is true in the setting we are interested in, namely that of families of smooth Fano manifolds.

\begin{thm}\label{invariance_eff}
 Let \(\Xcal \rightarrow B\) be  smooth family of Fano varieties, with \(B\) a smooth curve.
There exists a smooth curve \(B'\) and a surjective morphism \(B' \rightarrow B\) such that, setting \(\Xcal':=\Xcal\times_{B}B'\), for any \(\Lcal \in \Pic(\Xcal')\) one has that \(h^0(\Xcal'_s, \Lcal_{|\Xcal'_s})\) is independent of \(s\in B\). 
In particular,  for any \(s,t \in B\), if \(\Lcal_{|\Xcal_s}\) defines an effective divisor then so does \(\Lcal_{|\Xcal_t}\). 
\end{thm}

\begin{proof}
This is an immediate application of the much more general theorem \cite[Theorem 6.1]{Totaro2012}. As the author himself mentions, the case of his result we need here is less general and was already present in the literature  beforehand, see the works \cite{deFernex_Hacon2011} and \cite{Siu2002}.
\end{proof}

The invariance of dimensions of linear system as their line bundle deforms is a very strong property of families of Fano manifolds, which allows us to obtain two important corollaries. The first one was suggested to us by C. Voisin:

\begin{cor}\label{cor_irred_choice_deforms}
We put ourselves in the same setting as Theorem \ref{invariance_eff}. If \(|\Lcal_{|\Xcal_s}|\) admits a prime member, then so does \(|\Lcal_{|\Xcal_t}|\).
\end{cor}
\begin{proof}
Note that \(|\Lcal_{|\Xcal_s}|\) admits a member which is a prime divisor if and only if, for any two nontrivial effective line bundles \(L_1,L_2\) on \(\Xcal_{s}\) such that \(L_1\otimes L_2\simeq \Lcal_{|\Xcal_s}\), the multiplication map
\[\mu_{L_1,L_2} \colon |L_1|\times |L_2| \to |\Lcal_{|\Xcal_s}|\]
fails to be surjective. Indeed, this condition is clearly necessary as any prime divisor should lie outside the image of each \(\mu_{L_1,L_2}\). Conversely, if all the countably many Zariski-closed subsets \(\textup{im} (\mu_{L_1,L_2})\) were proper subsets of \(|\Lcal_{|\Xcal_s}|\), then by the theorem of Baire their union cannot be the whole \(|\Lcal_{|\Xcal_s}|\), hence there exists one member of \(|\Lcal_{|\Xcal_s}|\) which cannot be written as the sum of two nontrivial effective divisors.

We now remark that the morphism \(\mu_{L_1,L_2}\) has finite fibers, as there are only finitely many ways to write a fixed effective divisor as the sum of two effective subdivisors. Then, the surjectivity of \(\mu_{L_1,L_2}\) is equivalent to the dimensional condition
\begin{equation}\label{eq_condition_for_effectivity}  h^0(L_1)+h^0(L_2)-2=h^0(\Lcal_{|\Xcal_s})-1.
\end{equation}
We now prove the claim of the Corollary: suppose no element of \(|\Lcal_{|\Xcal_t}|\) is prime. Then there exists \(L_1,L_2 \in \Pic(\Xcal_t)\) such that \(L_1\otimes L_2\simeq \Lcal_{|\Xcal_t}\) and \eqref{eq_condition_for_effectivity} is satisfied. Up to basechange, \(L_1\) and \(L_2\) extend into families \(\Lcal_1,\Lcal_2\). 
By Theorem \ref{invariance_eff}, \eqref{eq_condition_for_effectivity} is still satisfied for \(\Lcal_{1|\Xcal_s}, \Lcal_{2|\Xcal_s}\). Since a line bundle is trivial if and only if both itself and its dual have a global section, which is a property that deforms well by Theorem \ref{invariance_eff}, we deduce that the line bundle \(\Lcal_1\otimes \Lcal_2\otimes \Lcal^\vee\) is trivial on all fibers as it was above \(t\) by hypothesis. 
In particular \(\Lcal_{1|\Xcal_s}\otimes \Lcal_{2|\Xcal_s}\simeq \Lcal_{|\Xcal_s}\), which by the above characterization proves \(|\Lcal_{|\Xcal_s}|\) admits no prime members.

\end{proof}
\color{black}

\begin{cor}\label{cor_spread_div}
Let \(\Xcal \rightarrow B\) be  smooth family of Fano varieties, with \(B\) a smooth curve. Let \(\Lcal\) be a line bundle on \(\Xcal\). For any two points \(b_1,b_2 \in B\) and two divisors \(D_i \in |\Lcal_{|\Xcal_{b_i}}|\), there exists a generically finite, surjective morphism \(B'\to B\) such that \(\Xcal'\coloneqq\Xcal \times_B B'\) admits a divisor \(\Dcal\subset \Xcal'\) flat over \(B'\) which satisfies \(\Dcal_{|X_i}=D_i\) for \(i=1,2\). 
\end{cor}

\begin{proof}
We consider the functor \(LinSym_{\Lcal/\Xcal/B}\) on \(B\)-schemes, whose \(T\)-points parametrize 
relative Weil divisors \(\Dcal \subset \Xcal\times_B T\) such that \(\Dcal_{|\Xcal_t}\in |\Lcal_{\Xcal_t}|\) for all \(t\in T\) (see \cite[Section 9.3]{fga} for more details).
By \cite[Theorem 9.3.13]{fga}, \(LinSym_{\Lcal/\Xcal/B}\) is represented by a projective bundle \(\pi \colon P\to B\), whose fiber over \(b\in B\) is the linear system \(|\Lcal_{|\Xcal_b}|\). 
Note that by Theorem \ref{invariance_eff}, this linear system has constant dimension as \(b\) varies in \(B\) (possibly after basechange). 
Then \(\pi \colon P\to B\) is a proper surjective morphism over an irreducible variety with fibers that are irreducible of constant dimension \(d\), hence \(P\) itself is irreducible (this is given as an exercise in \cite[12.4.D.]{rising_sea}): the idea is that since \(\pi\) is proper and surjective and \(B\) is irreducible, amongst the components \(P_i\) of \(P\) there must be one which we call \(P_\ell\) that dominates \(B\). One then proves that each component \(P_i\) is a union of fibers, which forces \(P_\ell=P\). Indeed, once \(i\) is fixed, by generic flatness and openness of finitely presented flat morphisms, the set \(\pi (P_i \setminus \bigcup_{j\neq i}P_j)\) contains a nonempty open \(V_i\). By irreducibility of the fibers, \(\pi^{-1}(y)\subset P_i\) for all \(y \in V_i\). Recall now that for any \(y \in \pi(P_i)\), one has
\[\dim \pi_{|P_i}^{-1}(y)\geq \dim P_i  - \dim \pi(P_i)\]
and that for some nonempty open \(W_i\subset B\) and any \(y \in W_i\), the above is an equality. Then for any \( y \in V_i\cap W_i\) we have \(\pi_{|P_i}^{-1}(y)=\pi^{-1}(y)\) and thus for any other \(y' \in \pi(P_i)\) we get
\[d =\dim \pi^{-1}(y)= \dim P_i - \dim \pi(P_i)\leq \dim \pi_{|P_i}^{-1}(y')\leq \pi^{-1}(y')=d.\]
By irreducibility of fibers this forces \(\pi_{|P_i}^{-1}(y')=\pi^{-1}(y')\) and concludes.

Now that we know \(P\) is irreducible, we can finish the proof of the Corollary by the well known fact (\cite[Example 3.3.5]{laz}) that there is an irreducible curve passing through any two points of a quasi-projective irreducible variety; applying this to \(P_{red}\), we deduce that there is a smooth proper curve \(B' \to P\) passing through the points \(D_1,D_2 \in P\). Composing this with \(P\to B\) defines the basechange \(B'\to B\) we are after and completes the proof.
\end{proof}
\subsection{Invariance of \texorpdfstring{\(\delta_X\)}{d_X}   for Fano families}
In this section we prove the main Theorem \ref{thm main}, namely that the Lefschetz Defect of a smooth Fano variety is invariant under smooth deformation.

Before beginning the proof, recall that for any complex manifold  \(X\) there exists a natural map
\[\Pic(X) \rightarrow H^2(X,\Z).\]
which sends \(L\) to its first chern class \(c_1(L)\). When \(X\) compact, for any \(D\) effective Cartier divisor one can define \([D]\) as discussed in Section \ref{subsec_another_char}: then one has that \([D]=c_1(\Ocal_X(D))\).

\begin{proof}[Proof of Theorem \ref{thm main}]
It suffices to show \(\delta_{X_1}\leq \delta_{X_2}\). By definition of \(\delta_{X_1}\), there exists a prime divisor \(D_1 \subset X_1\) such that \(\delta_{X_1}(D_1)=\delta_{X_1}\). 
Let \(\Xcal \rightarrow B\) be a smooth family and \(b_1, b_2 \in B\) such that \(\Xcal_{b_i}= X_i\).
Up to basechanging, by Theorem \ref{spreadPic} there exists a line bundle \(\Lcal\) on \(\Xcal\) such that \(\Lcal_{|X_i}=\cO_{X_1}(D_1)\). By Corollary \ref{cor_irred_choice_deforms}, since \(D_1\) is prime we may choose \(D_2 \in |\Lcal_{|X_2}|\) to be prime too. Up to another basechange, by Corollary \ref{cor_spread_div} we may pick a relative divisor \(\Dcal\subset \Xcal\) over \(B\) such that \(\Dcal_{|X_1}= D_1\), and \(\Dcal_{|X_2}=D_2\). Now by Corollary \ref{cor_another_N1defect} we have
\[\delta_{X_1}(D_1)=\dim K_{b_1}(D_1) \ \ \ \ , \ \ \ \ \delta_{X_2}(D_2)=\dim K_{b_2}(D_2),\]

where for every \(s\in B\) and \(D\subset \Xcal_s\) we defined
\[K_s(D)\coloneqq \Ker (H^2(\Xcal_s,\Q)\xrightarrow{\smallsmile [D]} H^4(\Xcal_s,\Q)).\]
Now we only need to understand how \(\dim K_s (\Dcal_s)\) varies as a function of \(s\in B\).
\color{black}
Fix some \(t \in B\). For a small enough \(U \subset B\) contractible analytic neighbourhood of \(t\), the submersion \(\Xcal_U \rightarrow U\) is \(C^\infty\)-trivial. In particular the restriction morphism \(H^2(\Xcal_U,\Q) \rightarrow H^2(\Xcal_s,\Q)\) is an isomorphism for all \(s \in U\). By the functoriality of taking first chern classes of line bundles on complex manifolds,
\(c_1(\Ocal(\Dcal))_{|\Xcal_s}=c_1(\Ocal(\Dcal_{|\Xcal_s}))=[\Dcal_{|\Xcal_s}]\) for all \(s \in U\). In particular, for any \(s \in U\) we have the following commutative diagram
\[\begin{tikzcd}
	{H^2(\Xcal_{t},\Q)} & {H^2(\Xcal_U,\Q)} & {H^2(\Xcal_s,\Q)} \\
	\\
	{H^4(\Xcal_{t},\Q)} & {H^4(\Xcal_U,\Q)} & {H^4(\Xcal_s,\Q)}
	\arrow["{ \smallsmile [\Dcal_{|\Xcal_{t}}]}"{description}, from=1-1, to=3-1]
	\arrow["\sim"', from=1-2, to=1-1]
	\arrow["\sim", from=1-2, to=1-3]
	\arrow["{ \smallsmile c_1(\Ocal_{\Xcal_U}(\Dcal_{|U}))}"{description}, from=1-2, to=3-2]
	\arrow["{ \smallsmile [\Dcal_{|\Xcal_s}]}"{description}, from=1-3, to=3-3]
	\arrow["\sim", from=3-2, to=3-1]
	\arrow["\sim"', from=3-2, to=3-3]
\end{tikzcd}.\]
The above diagram implies that \(K_t(\Dcal_{|\Xcal_t})\) and \(K_s(\Dcal_{|\Xcal_s})\) are identified via the isomorphism in cohomology induced by the diffeomorphism of the fibers \(\Xcal_{t} \simeq \Xcal_s\). 
Concatenating finitely many such local trivializations, we deduce 
\[\dim K_{b_1}(D_1)=\dim K_{b_2}(D_2).\]
Putting everything together, we conclude
\begin{equation*}\delta_{X_1}=\delta_{X_1}(D_1)=\dim K_{b_1}(D_1) = \dim K_{b_2}(D_2) = 
\delta_{X_2}(D_2)\leq \delta_{X_2} .\qedhere
\end{equation*}
\end{proof}

\section{Lefschetz defect on abelian varieties}\label{sec_defect_abelian}

In this section, we study the number \(\delta_A\) in the case \(A\) is an abelian variety. In particular, for any effective divisor \(D\subset A\), we observe that the number \(\delta_A(D)\) behaves in a way that heavily depends on the \textit{Iitaka dimension} of the line bundle \(\Ocal_A(D)\). 
\accapino

Recall the definition of Iitaka dimension: for \(L\) a line bundle on a normal algebraic variety \(X\),
\[\kappa(X,L):=\mathrm{min} \left\{d\in \N \; \bigg| \  \frac{h^0(X,L^{\otimes m})}{m^d} \ \mathrm{bounded \  as} \ m \rightarrow \infty \right \}.\]

If we denote by \(\phi_{L^{\otimes m}} : X \dashrightarrow \proj(H^0(X,L^{\otimes m}))\) the rational map induced by the linear system \(|L^{\otimes m}|\), and by \(W_{L^{\otimes m}}\subset \proj(H^0(X,L^{\otimes m})\) the closure of \( \mathrm{Im} \phi_{L^{\otimes m}}\), then 
\begin{equation}\label{eq kodaira dimension}
\kappa (X,L)= \dim W_m \ \ \ \mathrm{for } \ \ m>>0 .
\end{equation}

Let us fix a complex abelian variety \(A\) of dimension \(n\). Then the Iitaka dimension of the line bundle \(\Ocal_A(D)\) for \(D\subset A\) an effective divisor  manifests in a very explicit geometric way, which we now recall.
It is a remarkable basic fact about abelian varieties that for any effective divisor \(D \subset A\), the linear system \(|2D|\) is basepoint-free. Moreover, if 
\[A \xrightarrow{\varphi} B  \xrightarrow{\psi} \proj(|2D|)\] is the Stein factorization of \(\phi_{2D}\), then \(B\) itself is an abelian variety (by Proposition, page 88 of \cite{mumford_book}). Since \(\psi\) is finite, by (\ref{eq kodaira dimension}) we have
\[\kappa (A,\Ocal(D))=\dim B.\]
We now discuss what is the link between \(\delta_A(D)\) and \(\dim B\). This analysis will also show that the finer datum of the isogeny factors of \(B\) and their multiplicity in \(A\) determines \(\delta_A(D)\) even more precisely.

Recall that for any complex manifold \(X\), \(\NS(X)\) denotes the Neron-Severi lattice, i.e., the image of the first Chern class map \(\Pic(X)\to H^2(X,\Z)\). If \(X\) is smooth projective, then \(\NS(X)\otimes \Q\simeq N^1(X)_{\Q}\) as homological equivalence and numerical equivalence coincide up to torsion (\cite[Corollary 1.4.38]{laz}).

\begin{prop}\label{prop defect of D abelian}
Let \(A\) be an abelian variety, \(D\subset A\) a prime divisor and \(A \rightarrow B\) constructed as above.  Set \(b := \mathrm{dim} \;B\). 
Then we have one of the following cases:

\begin{itemize}
    \item If \(b\geq 3\), then \(\delta_A(D) =0\).
    \item If \(b=2\), then \(\delta_A(D) =\rho_B -1 \in \{0, \dots, 3\}\).
    \item If \(b=1\), meaning \(B\) is an elliptic curve which appears with multiplicity \(k\geq 1\) in the factorization up to isogeny of \(A\), then  
    
    \begin{equation}\label{eq CM cases}\delta_A(D)=
    \begin{cases}k & \mathrm{\ if } \ B \textrm{ does not have CM} \\
    2k-1 &  \mathrm{ \ if \ } B \textrm{ has CM } \end{cases}. \end{equation} 
\end{itemize}

\end{prop}
\begin{proof}
Note that by the discussion preceeding (\ref{remark alldivisors}) one has \(\delta_A(D)=\delta_A(2D)\), so up to substituting \(D\) with \(2D\), we may suppose that \(D\) is basepoint-free. This will not pose any problems later on  when we try to apply \eqref{eq another N1 defect}, by Remark \ref{rk_thanks_referee}.
Set \(C:= \Ker \, \varphi\), where \(\varphi\) is the first factor in the Stein factorization of the map associated to \(|D|\), as in the introduction to this Section. Then \(\varphi\) has connected fibers, so \(C\) is a connected abelian subvariety of \(A\). By the Poincaré splitting theorem, there exists another abelian subvariety \(T\subset A\) and an isogeny \(\mu\colon C\times T \rightarrow A\) fitting in the commutative diagram
\[\begin{tikzcd}
	{C\times T} & A \\
	T & B & {\proj(|D|)}
	\arrow["\mu", from=1-1, to=1-2]
	\arrow["{pr_T}"', shift right, from=1-1, to=2-1]
	\arrow["\varphi", from=1-2, to=2-2]
	\arrow[shift right, hook', from=2-1, to=1-1]
	\arrow["{(\varphi\circ \mu)_{|T}}"', from=2-1, to=2-2]
	\arrow["\psi"', from=2-2, to=2-3]
\end{tikzcd}.\]
In particular, the lower horizontal map is an isogeny as well, hence the pullback along finite maps \(N:= (\varphi \circ \mu)_{|T}^* \psi^* \Ocal_{\proj}(1)\) is an ample line bundle on \(T\), and 
\[pr_T^* N = \mu^* \varphi^* \psi^*\Ocal_{\proj(|D|)}(1)=\mu^* \Ocal_A(D).\]
In the commutative diagram
\[\begin{tikzcd}
	{N^1(C\times T)_\R} && {N^1(A)_\R} \\
	{H^4(C\times T ,\R)} && {H^4(A ,\R)}
	\arrow["{\smallsmile pr^*[N]}"', from=1-1, to=2-1]
	\arrow["{\mu^*}"', from=1-3, to=1-1]
	\arrow["{\smallsmile [D]}"', from=1-3, to=2-3]
	\arrow["{\mu^*}"', from=2-3, to=2-1]
\end{tikzcd},\]
the horizontal arrows are isomorphisms since induced by an isogeny, which has an inverse up to multiplication morphisms (\cite[Proposition 1.2.6]{birlan2004}). Then the two vertical arrows have isomorphic kernels, hence by (\ref{eq another N1 defect}) we get
\[\delta_A(D)=\dim K \mathrm{ \ \ where \ \ } K\coloneqq\Ker \Big( N^1(C\times T)_\R \xrightarrow{\smallsmile pr^*[N]} H^4(C\times T ,\R) \Big).\]

In other words, we reduced the computation of \(\delta_A(D)\) to the case where \(A\) is a product and \(\varphi\) is the projection onto one factor.

By the Künneth formula, we have the splitting
\begin{equation}\label{eq H2 of product}
    H^2(C\times T) = H^2(C) \oplus H^1(T)\otimes H^1(C) \oplus H^2(T) ,
\end{equation}
Recall that for \(M,N\) two \(\Z\)-Hodge structures of weight \(k\), the \((0,0)\) part of the natural weight \(0\) Hodge structure on \(\Hom_\Z(M,N)\) is the subgroup of morphism of Hodge structures \(\Hom_{\Z-HS}(M,N)\). For this reason, by Lefschetz \((1,1)\) theorem and by the correspondence between tori and weight \(1\) \(\Z\)-Hodge structures, taking the part of type \((1,1)\) of the above equality yields

\begin{equation}\label{eq N1 of product}
    \NS(C\times T) = pr_C^*\NS(C) \oplus \Hom(C^\vee, T) \oplus pr_T^*\NS(T) ,
\end{equation}

Since the map \(\smallsmile pr_T^* [N]\) preserves the Künneth decompositions  of domain and codomain, we can write \(K\) as
\[K= (K\cap pr_C^*NS(C)_\R) \oplus (K\cap \Hom(C^\vee, T)_\R )\oplus (K\cap pr_T^*NS(T)_\R).\]

But by Hard Lefschetz for the pair \((T,[N])\), we have some understanding of the dimension of each summand: since \(\dim T = \dim B= b\), and \(N\) is ample on \(T\), the map \[\smallsmile [N]: H^2(T,\R) \rightarrow H^4(T,\R)\] is injective as soon as \(b \geq 3\). Similarly, \(H^1(T,\R) \rightarrow H^3(T,\R)\) and \(H^0(T,\R) \rightarrow H^2(T,\R)\) are injective for \(b \geq 2\) and \(b\geq 1\) respectively.
Since all morphisms \(pr^*_T\) are injective, we get immediately that
\[b\geq3 \Rightarrow \dim_\R K =0,\]
\[b=2 \Rightarrow \dim K = \rho_T -1 .\]

Now let us suppose \(b=1\), meaning \(T\) is an elliptic curve. If \(O\in T\) denotes the neutral element of \(T\), we have \([C]= \varphi^* [O]\) as codimension 1 cycles on \(A\). The 0-cycle \([O]\) is an ample generator of \(\Pic \, T\), and \(\psi^* \Ocal_\proj (1)=\Ocal_T(m[O])\) where \(m=\deg \psi^* \Ocal_\proj (1)>0\). In conclusion, \[\Ocal_A(D)= \phi_D^* \Ocal_\proj(1) = \varphi^* \Ocal_T(m[O])= m \varphi^*\Ocal_T([O])=m \Ocal_A(C),  \]
hence \(D\) is linearly equivalent to \(m[C]\).

In particular, \(\delta_A(D)= \delta_{C\times T}(C)\), hence by (\ref{rk retraction}) and (\ref{eq N1 of product}) we get
\[\delta_A(D)= \rho_A - \rho_C=1 + \textup{rk}(\Hom(C^\vee, T))=1+\dim_\Q (\Hom_\Q(C^\vee,T)).\]
Since every abelian variety is isogenous to its dual, if \(k_T\) denotes the multiplicity of the elliptic curve \(T\) in the factorization up to isogeny of \(A\), then we have
\[\Hom_\Q(C^\vee, T)= \Hom_\Q (T, T^{k_T-1})= (k_t-1)\End_\Q(T).\]
Since \(k_T=k_B\), we obtain the formula (\ref{eq CM cases}).

The case \(b=0\) cannot occur, since that would mean \(\delta_X(D)= \rho_A\), which cannot happen as remarked in (\ref{range_for_defect}). Alternatively, note that \(C=\Ker \varphi \) is by construction (\cite{mumford_book} page 88) a connected component of \[H(D)=\{z\in A| t_z^*D=D\},\]
where \(t_z\) is the translation by \(z\), but this subset is never the whole \(A\) since for any choice of \(x \in D\), \(y \notin D\), then \(y-x \notin H(D)\).

\end{proof}

As the above proposition shows, it turns out that \(\delta_A\) is very often reached by \(\delta_A(D)\) where \(D\) is an abelian subvariety of codimension one, apart from certain cases where \(\delta_A(D)\) is low and determined by very specific simple abelian surfaces showing up as isogeny factors of \(A\). We can now prove Theorem \ref{thm_defect_of_abelians}, which we state again for the convenience of the reader.

\begin{thm*}(Theorem \ref{thm_defect_of_abelians}).
Let \(A\) be an abelian variety of dimension \(n\). Then \(\delta_A=0\) if and only if \(A\) does not contain elliptic curves nor abelian surfaces of Picard number greater than 1. If instead \(\delta_A>0\), then exactly one of the following is verified:
\begin{itemize}
\item \(A\sim C \times E\) where \(E\) is an elliptic curve, and if \(k_E \geq 1\) is its multiplicity as an isogeny factor of \(A\) then \[\delta_A=\begin{cases}
k_E & \mathrm{ \ if \ } E \textrm{ does not have CM }\\
2k_E -1 & \mathrm{ \ if \ } E \textrm{ has CM}
\end{cases}.\]
\item \(A\sim C\times S\) where \(S\) is a simple abelian surface of type II, \(\rho_S=3\) and \[\delta_A=2.\]
Furthermore, any elliptic curve showing up as an isogeny factor of \(A\) does so with multiplicity at most one.

\item \(A\sim C\times S'\) where \(S'\) is a simple abelian surface of type I or IV, \(\rho_{S'}=2\) and \[\delta_A=1.\]  Furthermore, there are no elliptic curves contained in \(A\).
\end{itemize} 

\end{thm*}

\begin{proof}[Proof of Theorem \ref{thm_defect_of_abelians}]
The statement in the case \(\delta_A=0\) follows immediately from Proposition \ref{prop defect of D abelian}. Suppose then \(\delta_A(D)=\delta_A >0\) for some effective divisor \(D\subset A\). Then \(b\leq 2\): if \(b=1\), we fall in the first case of the present Theorem, by the third point of Proposition \ref{prop defect of D abelian}. If instead \(b=2\), we have \(\delta_A =\delta_A(\varphi^* N)\in \{1,2,3\}\) for some \(N\) ample on \(B\). If \(B\) were not simple, meaning \(B\sim E\times E'\), and if \(L\) denotes the pullback of an ample line bundle of \(E\) onto \(A\), then \(\delta_A(L)\geq \rho_B-1=\delta_A(D)=\delta_A\) so \(\delta_A(L)=\delta_A\) and we fall again in the first case of the Theorem. 

So let us suppose we are not in the first case of the Theorem, which by the above argument means that \(b=2\) and \(B\) simple. The latter implies that \(\delta_A \in \{1,2\}\); once again by \(\delta_A(\varphi^* N)=\rho_B-1\), these two possible values correspond to the second and third cases of the Theorem. The type of the simple abelian surfaces follows from the restrictions on Picard rank imposed by the type, as in \cite[Proposition 5.5.7 ]{birlan2004}. The stated obstruction to the presence of elliptic curves is due to the working assumption of us not being in the first case of the Theorem.
\end{proof}

As an example, Theorem \ref{thm_defect_of_abelians} says that if \(A\) is an abelian threefold, then surely \(\delta_A\in \{0,1,2,3,5\},\) and we will have  the following cases (\(E_{RM}\) and \(E_{CM}\) denote respectively real-multiplication and complex-multiplication elliptic curves):

\begin{itemize}
\item \(\delta_A=5\) and  \(A\sim E_{CM}^3\).
\item \(\delta_A=3\) and  \(A\sim E_{RM}\times E_{CM}^2\).
\item \(\delta_A=2\) and  either \(A\sim S_{\textrm{II}}\times E\), where \(S_{\textrm{II}}\) is a simple abelian surface of type II (hence \(\rho_{S_{\textup{II}}}=3\)), or \(A\sim E_{RM}^2\times E\) with \(E\not\sim E_{RM}\).
\item \(\delta_A=1\) and  either\(A\sim E\times E' \times E''\) with all three curves non-isogenous, or \(A\sim S\times E\) with \(S\) simple not isogenous to \( S_{\textrm{II}}\) such as above. 
\item \(\delta_A=0\) and  \(A\) simple.
\end{itemize}

\noindent\textbf{Declarations}

\noindent \textbf{Conflict of interest:} On behalf of all authors, the corresponding author states that there is no conflict of interest

\noindent \textbf{Data availability statement:} Data sharing is not applicable to this article as no new data
were created or analyzed in this study.

\color{black}
\bibliographystyle{alpha}
\bibliography{sources} 

@Article{cas2012,
 Author = {{Casagrande, C.}},
 Title = {: {On} the {Picard} number of divisors in {Fano} manifolds},
 FJournal = {Annales Scientifiques de l'{\'E}cole Normale Sup{\'e}rieure. Quatri{\`e}me S{\'e}rie},
 Journal = {Ann. Sci. {\'E}c. Norm. Sup{\'e}r. (4)},
 Volume = {\textbf{45}},
 Number = {3},
 Pages = {363--403},
 Year = {2012},
 Language = {English},
 URL = {smf4.emath.fr/en/Publications/AnnalesENS/4_45/html/ens_ann-sc_45_363-403.php},
 zbMATH = {6109718},
 Zbl = {1267.14050}
}

@Article{cas2022,
 Author = {{Casagrande C.}},
 Title = {: {The} {Lefschetz} defect of {Fano} varieties},
 FJournal = {Rendiconti del Circolo Matem{\`a}tico di Palermo. Serie II},
 Journal = {Rend. Circ. Mat. Palermo (2)},
 ISSN = {0009-725X},
 Volume = {\textbf{72}},
 Number = {6},
 Pages = {3061--3075},
 Year = {2023},
 Language = {English},
 DOI = {10.1007/s12215-022-00846-4},
 zbMATH = {7725416}
}

@book{fga,
  title={: Fundamental Algebraic Geometry: Grothendieck's FGA Explained},
  author={{Fantechi B.} and {Göttsche L.} and {Illusie L.} and {Kleiman S.} and {Nitsure N.} and {Vistoli A.}},
  isbn={9780821842454},
  lccn={2005053614},
  series={Mathematical surveys and monographs},
  url={https://books.google.fr/books?id=UNxU24JAHDoC},
  year={2005},
  publisher={American Mathematical Society}
}

@article{cas_dru,
    author = {{Casagrande C.} and {Druel S.}},
    title = {: {Locally} Unsplit Families of Rational Curves of Large Anticanonical Degree on {Fano} Manifolds},
    journal = {IMRN},
    volume = {\textbf{2015}},
    number = {21},
    pages = {10756-10800},
    year = {2015},
    issn = {1073-7928},
    doi = {10.1093/imrn/rnv011},
    url = {https://doi.org/10.1093/imrn/rnv011},
    eprint = {https://academic.oup.com/imrn/article-pdf/2015/21/10756/6992045/rnv011.pdf},
}

@Article{deFernex_Hacon2011,
 Author = {{De Fernex  T.} and {Hacon C. D.}},
 Title = {: {Deformations} of canonical pairs and {Fano} varieties},
 FJournal = {Journal f{\"u}r die Reine und Angewandte Mathematik},
 Journal = {J. Reine Angew. Math.},
 ISSN = {0075-4102},
 Volume = {\textbf{651}},
 Pages = {97--126},
 Year = {2011},
 Language = {English},
 DOI = {10.1515/CRELLE.2011.010},
 zbMATH = {5865193},
 Zbl = {1220.14026}
}

@Article{Totaro2012,
 Author = {{Totaro B.}},
 Title = {: Jumping of the nef cone for {Fano} varieties},
 FJournal = {Journal of Algebraic Geometry},
 Journal = {J. Algebr. Geom.},
 ISSN = {1056-3911},
 Volume = {\textbf{21}},
 Number = {2},
 Pages = {375--396},
 Year = {2012},
 Language = {English},
 DOI = {10.1090/S1056-3911-2011-00557-2},
 zbMATH = {6026219},
 Zbl = {1248.14020}
}

@InCollection{siu2002,
 Author = {{Siu Y.-T.}},
 Title = {: {Extension} of twisted pluricanonical sections with plurisubharmonic weight and invariance of semipositively twisted plurigenera for manifolds not necessarily of general type},
 BookTitle = {Complex geometry. Collection of papers dedicated to Hans Grauert on the occasion of his 70th birthday},
 ISBN = {3-540-43259-0},
 Pages = {223--277},
 Year = {2002},
 Publisher = {Berlin: Springer},
 Language = {English},
 zbMATH = {1803782},
 Zbl = {1007.32010}
}

@InCollection{voisin1992,
 Author = {{Voisin C.}},
 Title = {: Sur la stabilit{\'e} des sous-vari{\'e}t{\'e}s lagrangiennes des vari{\'e}t{\'e}s symplectiques holomorphes.},
 BookTitle = {Complex projective geometry. Selected papers from the conference on projective varieties, 19-24 June 1989 in Trieste, Italy, and the conference on vector bundles and special projective embeddings, 3-16 July 1989 in Bergen, Norway},
 ISBN = {0-521-43352-5},
 Pages = {294--303},
 Year = {1992},
 Publisher = {Cambridge: Cambridge University Press},
 Language = {French},
 zbMATH = {125560},
 Zbl = {0765.32012}
}

@Book{goresky1988,
 Author = {{Goresky M.} and {MacPherson R.}},
 Title = {: Stratified {Morse} theory},
 FSeries = {Ergebnisse der Mathematik und ihrer Grenzgebiete. 3. Folge},
 Series = {Ergeb. Math. Grenzgeb., 3. Folge},
 ISSN = {0071-1136},
 Volume = {\textbf{14}},
 ISBN = {3-540-17300-5},
 Year = {1988},
 Publisher = {Berlin etc.: Springer-Verlag},
 Language = {English},
 zbMATH = {192849},
 Zbl = {0639.14012}
}

@Article{shenyin2022,
 Author = {{Shen J.} and {Yin Q.}},
 Title = {: Topology of {Lagrangian} fibrations and {Hodge} theory of hyper-{K{\"a}hler} manifolds},
 FJournal = {Duke Mathematical Journal},
 Journal = {Duke Math. J.},
 ISSN = {0012-7094},
 Volume = {\textbf{171}},
 Number = {1},
 Pages = {209--241},
 Year = {2022},
 Language = {English},
 DOI = {10.1215/00127094-2021-0010},
 zbMATH = {7467783},
 Zbl = {1490.14019}
}

@article {Deligne1974,
    AUTHOR = {{Deligne  P.}},
     TITLE = {: {Th\'eorie} de {H}odge. {III}},
   JOURNAL = {Inst. Hautes \'Etudes Sci. Publ. Math.},
  FJOURNAL = {Institut des Hautes \'Etudes Scientifiques. Publications
              Math\'ematiques},
      YEAR = {1974},
     PAGES = {5--77},
      ISSN = {0073-8301,1618-1913},
   MRCLASS = {14C30 (14F15)},
  MRNUMBER = {498552},
MRREVIEWER = {J.\ H. M. Steenbrink},
       URL = {http://www.numdam.org/item?id=PMIHES_1974_44_5_0},
}

@Book{mumford_book,
 Author = {{Mumford D.}},
 Title = {: Abelian varieties. {With} appendices by {C}. {P}. {Ramanujam} and {Yuri} {Manin}.},
 Edition = {Corrected reprint of the 2nd ed. 1974},
 ISBN = {978-81-85931-86-9},
 Year = {2008},
 Publisher = {New Delhi: Hindustan Book Agency/distrib. by American Mathematical Society (AMS); Bombay: Tata Institute of Fundamental Research},
 Language = {English},
 zbMATH = {5358709},
 Zbl = {1177.14001}
}

@Article{cas2024,
 Author = {{Casagrande C.}},
 Title = {: {Fano} 4-folds with {{\(b_2>12\)}} are products of surfaces},
 FJournal = {Inventiones Mathematicae},
 Journal = {Invent. Math.},
 ISSN = {0020-9910},
 Volume = {\textbf{236}},
 Number = {1},
 Pages = {1--16},
 Year = {2024},
 Language = {English},
 DOI = {10.1007/s00222-024-01236-6},
 zbMATH = {7813237},
 Zbl = {1539.14087}
}

@book{birlan2004,
  title={: {Complex} Abelian Varieties},
  author={{Birkenhake C.} and {Lange H.}},
  isbn={9783540204886},
  lccn={2004045272},
  series={Grundlehren der mathematischen Wissenschaften},
  url={https://books.google.fr/books?id=MOW2gEP7HIkC},
  year={2004},
  publisher={Springer Berlin Heidelberg}
}

@Article{cas_romano_secci2022,
 Author = {{Casagrande C.} and {Romano E. A.} and {Secci S. A.}},
 Title = {: {Fano} manifolds with {Lefschetz} defect 3},
 FJournal = {Journal de Math{\'e}matiques Pures et Appliqu{\'e}es. Neuvi{\`e}me S{\'e}rie},
 Journal = {J. Math. Pures Appl. (9)},
 ISSN = {0021-7824},
 Volume = {\textbf{163}},
 Pages = {625--653},
 Year = {2022},
 Language = {English},
 DOI = {10.1016/j.matpur.2022.05.016},
 zbMATH = {7541878},
 Zbl = {1498.14110}
}

@article{Secci2023,
   title={: {Fano} fourfolds having a prime divisor of {Picard} number 1},
   volume={\textbf{23}},
   ISSN={1615-715X},
   url={http://dx.doi.org/10.1515/advgeom-2023-0002},
   DOI={10.1515/advgeom-2023-0002},
   number={2},
   journal={Adv. Geom.},
   publisher={Walter de Gruyter GmbH},
   author={{Secci  S.A.}},
   year={2023},
    pages={267–280}}

@article{Della_Noce_2012,
   title={: On the {Picard} Number of Singular {Fano} Varieties},
   volume={\textbf{2014}},
   ISSN={1073-7928},
   url={http://dx.doi.org/10.1093/imrn/rns236},
   DOI={10.1093/imrn/rns236},
   number={4},
   journal={IMRN},
   publisher={Oxford University Press (OUP)},
   author={Della {Noce G.}},
   year={2012},
pages={955–990} }

@book{laz,
 author = {Lazarsfeld, Robert},
 title = {Positivity in algebraic geometry. {I}. {Classical} setting: line bundles and linear series},
 fseries = {Ergebnisse der Mathematik und ihrer Grenzgebiete. 3. Folge},
 series = {Ergeb. Math. Grenzgeb., 3. Folge},
 issn = {0071-1136},
 volume = {48},
 isbn = {3-540-22533-1},
 year = {2004},
 publisher = {Berlin: Springer},
 language = {English},
 zbMATH = {2134816},
 Zbl = {1093.14501}
}

@book{rising_sea,
 author = {Vakil, Ravi},
 title = {The rising sea. {Foundations} of algebraic geometry},
 isbn = {978-0-691-26866-8; 978-0-691-26867-5; 978-0-691-26868-2},
 year = {2025},
 publisher = {Princeton, NJ: Princeton University Press},
 language = {English},
 zbMATH = {7961837}
}

\end{document}